\documentclass[10pt]{amsart}
\usepackage{amssymb,latexsym,amscd,amsmath,amsthm,manfnt,tikz}
\usepackage{mathdots}
\usepackage{xcolor}
\usepackage{hyperref}
\hypersetup{colorlinks=true,linkcolor=blue, citecolor=blue}

\usepackage[matrix,arrow]{xy}
\input{amssym.def}
\input{xy}

\xyoption{all}

\theoremstyle{definition}

\numberwithin{equation}{subsection} 
\newtheorem{guess}{theorem}[subsection]

\newtheorem{rem}[guess]{Remark}
\newtheorem{defi}[guess]{Definition}
\newtheorem{ex}[guess]{Example}
\newtheorem{thm}[guess]{Theorem}

\newtheorem{prop}[guess]{Proposition}
\newtheorem{Cor}[guess]{Corollary}

\newcommand{\gfr}{\mathfrak{g}}

\newcommand{\cO}{\mathcal{O}}

\newcommand{\cE}{\mathcal{E}}
\newcommand{\cG}{\mathcal{G}}

\newcommand{\cH}{\mathcal{H}}
\newcommand{\cM}{\mathcal{M}}
\newcommand{\cA}{\mathcal{A}}

\newcommand{\cR}{\mathcal{R}}

\newcommand{\cP}{\mathcal{P}}

\newcommand{\Sym}{\mathrm{Sym}}

\newcommand{\Spec}{\mathrm{Spec}}

\newcommand{\bsm}{\boldsymbol}
\newcommand{\lra}{\longrightarrow}

\newcommand{\ra}{\rightarrow}
\newcommand{\ol}{\overline}

\newcommand{\ms}{\mapsto}

\newcommand{\RR}{\mathbb{R}}

\newcommand{\PP}{\mathbb{P}}

\newcommand{\ZZ}{\mathbb{Z}}
\newcommand{\GG}{\mathbb{G}}

\newcommand{\CC}{\mathbb{C}}

\newcommand{\Hom}{\mathrm{Hom}}

\newcommand{\Id}{\mathrm{Id}}

\newcommand{\SL}{\mathrm{SL}}

\begin{document}

\title{On homomorphisms of $\pi_{1}(\mathbb P^1 - \cR)$ into compact semisimple groups}

\author{Vikraman Balaji}
\address{Chennai Mathematical Institute, H1 Sipcot IT Park, Siruseri Tamil Nadu 603103, India}
\email{balaji@cmi.ac.in}
\author[Y. Pandey]{Yashonidhi Pandey}
\thanks{The support of Science and Engineering Research Board under Mathematical Research Impact Centric Support File number: MTR/2017/000229 is gratefully acknowledged. We thank ICTS, Bangalore for excellent hospitality where this work was finally accomplished.}
\address{ 
Indian Institute of Science Education and Research, Mohali Knowledge city, Sector 81, SAS Nagar, Manauli PO 140306, India}
\email{ ypandey@iisermohali.ac.in, yashonidhipandey@yahoo.co.uk}

\begin{abstract} The aim of this paper is to give verifiable criteria for the existence of {\em irreducible} homomorphisms of $\pi_{1}(\mathbb P^1 - \cR)$ into compact semisimple groups, for a finite subset $\cR$ such that the conjugacy classes of the images of lassos around the marked points are fixed. By a theorem in \cite{bs}, this question reduces into one of giving verifiable criteria for the existence of stable $\cG$-torsors on $\mathbb{P}^1_{\mathbb{C}}$, where $\mathcal{G} \ra \mathbb{P}^1_\mathbb{C}$ is a  a Bruhat-Tits group scheme. \end{abstract}
\maketitle
\small
\tableofcontents
\normalsize


\section{Introduction}
Let $G$ be an almost simple simply connected algebraic group over $\CC$. Let $K_G$ denote its maximal compact subgroup. Let $\cR \subset \mathbb P^1$ be a finite subset of distinct points.  Recall that the fundamental group 
$\pi_{1}(\mathbb P^1 - \cR)$ is a free group on $s = \|\cR\|$-number of generators $\gamma_1, \ldots, \gamma_s$ such that $\gamma_1\ldots \gamma_s = 1$. Recall that a subset $H \subset K_G$ is called {\em irreducible} if the $\{Y \in \text{Lie}(G) \mid \text{ad} h(Y) = Y, \forall h \in H\} = \text{centre of} \, \text{Lie}(G) = 0$ and a homomorphism $\rho: \pi_{1}(\mathbb P^1 - \cR) \to K_G$ is called {\em irreducible} if the image $\rho(\pi_{1}(\mathbb P^1 - \cR)) \subset K_G$ is irreducible (see Ramanathan \cite{ramanathan}). For $i \geq 3$, let $\{C_i| 1 \leq i \leq s \}$ denote a prescribed set of conjugacy classes in $K_G$.  The aim of this paper is to give verifiable criteria for the existence of {\em irreducible} homomorphisms $\rho: \pi_{1}(\mathbb P^1 - \cR) \to K_G$, such the conjugacy class of $\rho(\gamma_j)$ lies in $C_j$.

 The multiplicative Horn problem asks whether there exists a set of lifts $\{c_i \in C_i \}$ satisfying $\prod c_i=1$. Teleman and Woodward \cite{tw} gave numerical criteria for this problem for arbitrary $G$, and earlier for $G= SU_n$ such a criteria was obtained independently by Agnihotri-Woodward \cite{agni} and P.Belkale \cite{belkale}.

Our question of the existence of an irreducible $\rho$ thus translates to the   Deligne-Simpson problem \cite{kostov2003}, which asks whether there exists a set of lifts $\{c_i \in C_i\}$ satisfying $\prod c_i=1$ {\em that also form an irreducible set}. It should be noted however that the general Deligne-Simpson problem asks the same question when the conjugacy classes are not necessarily constrained to take values in the maximal compact subgroup.

Recall that by Balaji-Seshadri \cite[Thm 8.1.7, Cor 8.1.8]{bs} the question of existence of such an irreducible set of lifts is equivalent to  the existence of a stable torsor under a suitable Bruhat-Tits group scheme on $\mathbb{P}^1_{\mathbb{C}}$. When $K_G=U_n$, this is classical and by Mehta-Seshadri \cite{ms} this problem is equivalent to the existence of a stable parabolic vector bundle on $\mathbb{P}^1_{\mathbb{C}}$.

From the perspective of the root system of $G$, we recall that conjugacy classes in $K_G$ are parametrized precisely by the points of the Weyl alcove ${\mathbf{a}_0}$ (cf \cite[Page 151]{morgan}). In this setting, our 
aim  is to give numerical verifiable criteria for the existence of such an irreducible set of lifts in terms of the points of the Weyl alcoves determined by the $\{C_i\}$. More precisely, let $\ol{\mathbf{a}_0}$ denote the points  of the {\it closed} Weyl  alcove. In this note we want to describe the stable polytope $\Delta^s \subset {\ol{\mathbf{a}_0}}^{|\cR|}$ i.e. the set of points ${\boldsymbol{\theta}}=\{\theta_x\}_{x \in \cR} \in \ol{\mathbf{a}_0}^{|\cR|}$ such that there exists a stable parahoric torsor on $\PP^1$ with weight ${\boldsymbol{\theta}}$. One defines the semistability polytope $\Delta^{ss}$ similarly.

Such criteria were  obtained by I.Biswas \cite{biswas} for $U_2$, by P.Belkale \cite{belkale} for $SU_n$ and by Y.Pandey \cite{me} for the  maximal compact subgroups of $SO_n(\CC)$ and $Sp_{2n}(\CC)$.

Returning to the setting of \cite{bs}, let $\cM(\cG)$ denote the moduli stack of $\cG$-torsors on $\PP^1_{\mathbb{C}}$ where $\cG$ is a Bruhat-Tits group scheme on $\PP^1_{\mathbb{C}}$ at a fixed set of marked points on $\PP^1_{\mathbb{C}}$ given by a choice of weights $\boldsymbol{\theta}$. Let $\cM(\cG^{^{\tt I}})$ be the moduli stack of torsors with Iwahori structures at these marked points. Recall that the points of $\cM(\cG^{^{\tt I}})$ can be viewed as principal $G$-bundles with parabolic structures given by the Borel $B$ at the markings, analogous to vector bundles with full-flag parabolic structures. Under this identification, the trivial $G$-bundle with $B$-structures can be seen as points of $\cM(\cG^{^{\tt I}})$.  These fit in together in the following Hecke-modification diagram  constructed in (\ref{Heckediag})
\begin{equation} \label{Heckediag}
\xymatrix{
& \cM_X(\cG^{{\tt I}}) \ar[ld]^{p} \ar[rd]^{q} & & \\
\cM_X(\cG) & &  \cM_X(G)
}
\end{equation}

Given a choice of weights $\boldsymbol\theta$ for $\cG$ we show how one can derive an extended set of weights $\boldsymbol\theta^{^{ext}}$ for $\cG^{{\tt I}}$ such that the stability condition for a $\cG$ torsor $\cE$ with weights $\boldsymbol\theta$  becomes equivalent to an intrinsic stability condition for all the $\cG^{{\tt I}}$ torsors with weights $\boldsymbol\theta^{^{ext}}$ sitting above $\cE$ under the map $p$. As we noted above, since these objects are now in $\cM(\cG^{^{\tt I}})$, the stability condition can be seen as   a $\boldsymbol\theta^{^{ext}}$-stability condition on the underlying parabolic $G$-bundle.  The technical heart of this paper is the following:
\begin{thm} The open sub-stack $\cM(\cG)^{^s}$ of $\cM(\cG)$ consisting of stable torsors is non-empty if and only if the trivial $G$-bundle with $B$-structures and weight $\boldsymbol\theta^{^{ext}}$ is stable as a point of $\cM(\cG^{^{\tt I}})$. 
\end{thm}
The setting is as in \cite{tw}, and  
the sought-after criterion now gets translated (cf Corollary \ref{stablepolytope}) into one in terms of Gromov-Witten numbers as it happens in \cite{tw}.

Let $\cE^{\tt I}$ be the {\em trivial $G$-bundle} with parabolic structures of the full-flag type, i.e. $B$-structures at the marked points $\cR$. For a parabolic subgroup $P \subset G$, let $\cE^{\tt I}_{_P}$ be a reduction of structure group to $P$. We then have an inclusion of Lie algebra bundles $\cE^{\tt I}_{_P}(\mathfrak p) \subset \cE^{\tt I}(\gfr)$. Observe that the associated Lie algebra bundle $\cE^{\tt I}(\gfr)$ gets canonical parabolic structures at the marked points (these will not be full-flag types though). We denote this Lie algebra bundle with parabolic structures by $\cE^{\tt I}(\gfr)_*$. The sub-bundle $\cE^{\tt I}_{_P}(\mathfrak p)$ gets the canonical induced parabolic structures and we have similarly $\cE^{\tt I}_{_P}(\mathfrak p)_* \subset \cE^{\tt I}(\gfr)_*$. Since $G$ is semisimple, it is clear that $pardeg(\cE^{\tt I}(\gfr)_*) = 0$. 

{ Say a $P$-reduction $\cE^{\tt I}_{_P} \subset \cE^{\tt I}$ is {\underline {of the minus 1 type}} if the parabolic degree of the quotient $\cE^{\tt I}(\gfr)_*/\cE^{\tt I}_{_P}(\mathfrak p)_*$ is $0$ and further, the degree of the vector bundle underlying the quotient $\cE^{\tt I}(\gfr)_*/\cE^{\tt I}_{_P}(\mathfrak p)_*$ is $-1$}. This condition naturally fits into the setting where Gromov-Witten numbers can be used to quantify it.

\begin{thm} \label{mt2} A point $\boldsymbol{\theta} \in \Delta^{ss}$  lies in $\Delta^{s}$ if and only if the trivial $G$-bundle $\cE^{\tt I}$ with generic $B$-structures and marking $\boldsymbol{\theta^{^{ext}}}$ does not have any $P$-reduction $\cE^{\tt I}_{_P}$  {\sl of the minus 1 type}.
\end{thm}
Corollary \ref{newvercrit} to the above theorem gives a new verifiable criterion for the main question of this paper.

\subsubsection{Some remarks on the far wall and stability } We conclude by giving a few words of justification why the far wall cannot be avoided for the stability question (although it can be avoided for semistability and $G=\text{GL}(n)$). Let us mention some difficulties. Firstly, in \cite[\S 7]{me} some examples of stable parahoric symplectic and (special  orthogonal)  torsors are shown to lie on the product of far walls.  Secondly, Belkale has shown that $(\Delta^{ss})^\circ \subset \Delta^{s}$ (cf \cite[Prop 7.0.5]{me}). So they have the same closures in $\ol{\mathbf{a}_0}^{|\cR|}$. Also the origin in $\ol{\mathbf{a}_0}^{|\cR|}$ lies in $\Delta^{ss} \setminus \Delta^s$ because it corresponds to the case of principal $G$-bundles. 
 Lastly, to the best of our knowledge, no argument like the one by Meinrenken and Woodward for semistability is known for stability which may allow one to restrict oneself to $\mathbf{a}_0^{|\mathcal{R}|}$. Thus,  it does not seem possible to reduce the problem of determining  $\Delta^s$ to the case of generic weights. We  are forced to consider the closure $\ol{\mathbf{a}}_0^{|\cR|}$ fully and directly. Now weights on the far wall of $\ol{\mathbf{a}_0}$ correspond to strictly parahoric (non-parabolic)  torsors under Bruhat-Tits group schemes and these need to be reckoned with. 



 
\subsection{Layout} We develop notions over a general smooth projective curve $X$ over an algebraically closed field $k$ of arbitrary characteristic throughout the paper and specialize to the case $X=\PP^1$ and $k=\CC$ only to prove the main theorems. In \S \ref{modulistack}, we explain our basic set-up. Then after recalling the main consequences of our construction,  we prove the main theorems in \S \ref{mtsection} and \S \ref{deformation}. The introduction of \S \ref{sshm} explains the main constructions of the paper. 
 
\section{The moduli stack $\cM_X(\cG)$} \label{modulistack}
\subsubsection{Local group theoretical data of parahoric group schemes} \label{lgpthedata}
We will write $k$ instead of $\CC$  whenever the results we use holds for an algebraically closed field of arbitrary characteristic.
 Let $A:=k[[t]]$ and $K:=k((t))=k[[t]][t^{-1}]$, where $t$ denotes a uniformizing parameter.
Let $G$ be a {\em semisimple simply connected affine algebraic
group} defined over $k$. We now want to consider the group $G(K)$. 


We shall fix a maximal torus $T \,\subset\, G$ and let
$Y(T)\,=\, \Hom(\GG_m,\, T)$ denote the group of all
one--parameter subgroups of $T$. For each maximal torus $T$ of $G$, the {\it
standard affine apartment} $\cA_T$  is an affine space under $Y(T)
\otimes_\ZZ \RR$. We may identify $\cA_T$ with $Y(T) \otimes_\ZZ \RR$
(see \cite[\S~2]{bs}) by choosing a point $v_0 \in \cA_T$. This $v_0$ is also called an origin. For a root $r$ of $G$ and an integer $n \in \ZZ$, we get an affine functional
\begin{equation} \label{afffunc} \alpha=r + n : \cA_T \ra \RR, x \ms r(x -v_0) +n.
\end{equation}
These are called the {\it affine roots} of $G$. For any point $x \in \cA_T$, let $Y_x$ denote the set of affine roots vanishing on $x$. For an integer $n\geq 0$, define
\begin{equation} \label{facetdefn}
\cH_n=\{x \in \cA_T | |Y_x|=n \}.
\end{equation}
A facet $\sigma$ of $\cA_T$ is defined to be a connected component of $\cH_n$ for some $n$. The dimension of a facet is its dimension as a real manifold. 

Let $R\,=\,R(T,G)$ denote the root system of $G$ (cf. \cite[p. 125]{springer}). Thus for
every $r \,\in\, R$, we have
the root homomorphism $u_r \,: \, \GG_a\,\lra\, G$ \cite[Proposition 8.1.1]{springer}. 
{\it In this paper $\Theta$ will always be either a {\em facet} or a point of $\cA_T$}. By \cite[Section 1.7]{bt} we have an affine flat smooth group scheme $\cG_\Theta\,\lra\, \Spec(A)$ called the {\it parahoric group scheme} associated to $\Theta$.  The group scheme $\cG_\Theta$ is uniquely determined by its
$A$--valued points. For a facet $\sigma \subset \cA_T$, let  $\cG_\sigma \ra \Spec(A)$ be the parahoric group scheme defined by $\sigma$.


\subsubsection{Alcove} \label{alcove}
We choose a Borel $B$ in $G/k$ containing $T$. This determines a choice of positive roots.  Let 
$\mathbf{a}_0$ denote the unique alcove in $\cA_T$ whose closure  contains $v_0$ and is contained in the finite Weyl chamber determined by positive simple roots. The affine walls defining $\mathbf{a}_0$ determine a set $\mathbf{S}$ of simple {\it affine roots}. We will denote these simple roots by the symbols $\{\alpha_i\}$.

\subsubsection{The Bruhat-Tits group scheme} \label{gpsch} For an arbitrary closed point $y \in X$ let $\mathbb{D}_y:=\Spec(\hat{\mathcal{O}_y})$, let $K_y$ be the quotient field of $\hat{\cO_y}$ and let $\mathbb{D}^\circ_y=\Spec(K_y)$. Let $\cR \subset X$ be a non-empty finite set of closed points. For each $x \in \cR$, we choose a facet $\sigma_x \subset \cA_T$. Let $\cG_{\sigma_x} \ra \mathbb{D}_x$ be the parahoric group scheme corresponding to $\sigma_x$. Let $X^\circ = X \setminus \cR$. {\it In this paper,  by a Bruhat-Tits group scheme $\cG \ra X$ we shall mean that  $\cG$ restricted to ${X^\circ}$ is isomorphic to $ X^\circ \times G$, and for any closed point $x \in X$,  $\cG$ restricted to $\mathbb{D}_x$ is a parahoric group scheme $\cG_{\sigma_x}$ such that the gluing functions take values in $Mor(\mathbb{D}^\circ_x,G)=G(K_x)$.} This is also the set-up of \cite[Defn 5.2.1]{bs} (see also \cite{heinloth} and \cite{zhu}). {\it We also suppose that the facets $\{\sigma_x\}_{x \in \cR}$ lie in the closure of $\mathbf{a}_0$.} It can easily be seen that for proving the main results of this paper, the general case of arbitrary facets reduces to this case.

Let $\mathbb{D} \subset X$ be an arbitrary formal disc about a point. Let $\cG_{_{\mathbb{D}}}$ be the restriction of $\cG$ to the disc. The  $\hat{\cO}$-points $\cG_{_{\mathbb{D}}}(\hat{\cO})$ gives subgroups of 
$\cG_{_{\mathbb{D}}}(K) = G(K)$ and these are called the parahoric subgroups of $G(K)$. The subgroup $G(\hat{\cO}) \subset G(K)$ is an example of a maximal parahoric subgroup. We have a natural evaluation map $ev:G(\hat{\cO}) \to G(\mathbb C)$ and the inverse image $\mathtt I := ev^{^{-1}}(B)$ of the standard Borel subgroup $B \subset G$ is called the standard Iwahori subgroup. Observe that any parahoric subgroup $\cG_{_{\mathbb{D}}}(\hat{\cO})$ contains a $G(K)$-conjugate of the standard Iwahori subgroup $\mathtt I$. The group scheme $\cG^{^{\mathtt I}}_{_{\mathbb{D}}}$ such that $\cG^{^{\mathtt I}}_{_{\mathbb{D}}}(\hat{\cO}) = \mathtt I$ is called the standard Iwahori group scheme. 


\begin{rem} \label{gluingfn} The group scheme  $\cG$ depends on the gluing data. But if $\cG$ and $\cG'$ are two parahoric group schemes on $X$ which differ only in their gluing data, then it is straightforward to check that the stacks $\cM_X(\cG)$ and $\cM_X(\cG')$ are isomorphic. For this reason, we fix one gluing data to get $\cG$ and work with this. \end{rem}

Let $\cG^{^{\mathtt I}} \ra X$ (resp. $\cG^{'^{\mathtt I}} \ra X$) be the group scheme obtained by gluing $X^\circ \times G$ with $\mathcal{G}^{^{\mathtt I}}$ at each parabolic point $x \in \cR$ using the same gluing functions as $\mathcal{G}$ (resp. $X \times G$). The inclusions $\mathtt I \subset \cG_{_{\mathbb{D}}}(\hat{\cO})$ (resp. $\mathtt I \subset G(\hat{\cO})$) induce morphisms of group schemes $\cG^{^{\mathtt I}} \to \cG$ (resp. $\cG^{'^{\mathtt I}} \to X \times G$) over the whole of $X$. 
\subsubsection{Parahoric torsors} \label{pt}

Let $\cG \ra X$ be a group scheme as in \S \ref{gpsch}. A {\it quasi-parahoric} torsor $\cE$ is a $\cG$--torsor on $X$.  This means that $\cE \times_X \cE \simeq \cE \times_X \cG$ and there is an action map $a: \cE \times_X \cG \ra \cE$ which satisfies the usual axioms for principal $G$-bundles. A {\it parahoric torsor} is a pair $(\cE\, , {\boldsymbol\theta})$ consisting of the pair of a quasi-parahoric torsor and  weights ${\boldsymbol\theta}\,=\, \{\theta_x| x \in \cR \} \in (Y(T) \otimes \RR)^m$ such that $\theta_x$ lies in the facet $\sigma_x$ (cf \S \ref{gpsch}) and $m=|\cR|$. Let $\cM_X(\cG)$ denote the moduli stack of $\cG$-torsors on $X$. The natural morphisms of group schemes seen above induces the following morphisms of stacks:
\begin{equation}
\cM_X(\cG) \leftarrow \cM_X(\cG^{^{\mathtt I}}) \stackrel{\ref{gluingfn}}{\simeq} \cM_X(\cG^{'^{\mathtt I}}) \ra \cM_X(G).
\end{equation}

In particular, the morphism $ \cM_X(\cG^{{\tt I}})  \ra \cM_X(G)$
induced by the morphism $\cG^{^{\mathtt I}} \to X \times G$ can be viewed as follows. The points of the stack $\cM_X(\cG^{{\tt I}})$ are   $G$-bundles on $X$ with $B$-structures at the marked points $\cR$. The morphism $\cM_X(\cG^{{\tt I}}) \ra \cM_X(G)$ forgets the $B$-structures. Thus, $\cM_X(\cG^{{\tt I}})$ seen from the standpoint of $\cM_X(G)$ is the analogue of the moduli stack of vector bundles with full-flag structures at the marked points.  {\sl We will call morphisms in the  diagram (\ref{Heckediag}) as Hecke-modification}.

Although in the literature a sequence of flip-flop is called a Hecke-modification, but we wish to emphasize that often  only a {\it single} morphism as above will be required for the proofs in this paper. For the usual case of  parabolic vector bundles these one-step modification morphisms  correspond to usual  Hecke-modifications. For torsors, we will call them both by the same name following Balaji-Seshadri \cite{bs}.

\section{Main Theorem} \label{mtsection}
In this section we suppose that $X=\PP^1$. We deduce the main theorem of the paper using general results proved in later sections. Let us summarize these.


 Let $(\cE,{\boldsymbol{\theta}})$ be a parahoric torsor. Under $p: \cM_X(\cG^{{\tt I}}) \ra \cM_X(\cG)$, let $\cE^{{\tt I}}$ be an arbitrary $\cG^{{\tt I}}$-torsors that maps to $\cE$ and consider the  Hecke-modification  diagram \eqref{Heckediag}. Given a parabolic vector bundle with possibly partial flags, Belkale \cite{belkale} makes a   construction called {\it completing flags}. In an analogous fashion (with a somewhat involved ``parahoric" adaptation),   in \S \ref{sshm} we explain how after choosing  any finer facet $\mathbf{a}^x_{Ad} \subset \mathbf{a}$, in whose closure $\theta_x$ lies,  ${\boldsymbol{\theta}}$ may be extended as a weight ${\boldsymbol{\theta}}^{^{ext}}$ on $\cE^{{\tt I}}$. 
We extend the definition of (semi)stablity for such objects and call this construction {\it extending weights}. Then in Proposition  \ref{stabilityExtendingweights} we show that $(\cE,{\boldsymbol{\theta}})$ is stable if and only if $(\cE^{{\tt I}}, {\boldsymbol{\theta}}^{^{ext}}, \{\mathbf{a}^x_{Ad} \})$ is stable as a extended weight parahoric torsor. Now we view $\cE^{{\tt I}}$ as a $G$-bundle together with additional parabolic structure at $\cR$.


\begin{thm} \label{mt} With notations as above, the open sub-stack $\cM(\cG)^{^s}$ of $\cM(\cG)$ consisting of stable torsors with weight $\bsm{\theta}$ is non-empty if and only if the trivial $G$-bundle with generic $B$-structures and weight $\boldsymbol{\theta^{^{ext}}}$ is stable. 
\end{thm}
\begin{proof} For  a given weight ${\boldsymbol{\theta}}$ let us choose for each $x \in \cR$ a finer facet $\mathbf{a}^x_{Ad} \subset \mathbf{a}_0$ in whose closure $\theta_x$ lies. By Proposition \ref{stabilityExtendingweights}, there exists a stable torsor $(\cE,{\boldsymbol{\theta}})$ if and only if $(\cE^{{\tt I}},{\boldsymbol{\theta}}^{^{ext}}, \{ \mathbf{a}^x_{Ad} \})$ is stable in the sense of Definition \ref{defnExtendingsemistability}. Under the Hecke modification diagram (\ref{Heckediag}) we may view $\cE^{{\tt I}}$ rather as a parabolic $G$-bundle with weights given by $\boldsymbol{\theta^{^{ext}}}$.


 Since $G$ is simply-connected, so by \cite[Thm 7.4]{ramdef} the principal $G$-bundle $E$ underlying $\cE^{{\tt I}}$ may be put in a family $\mathbf{E} \ra \PP^1 \times T$ where the generic bundle is the trivial $G$-bundle and $T$ is affine. Consider $T_1:=T \times_{\cM_X(G)} \cM_X(\cG^{{\tt I}})$ corresponding to the classifying map $T \ra \cM_X(G)$ of $\mathbf{E}$.  Using the finer facets $\{\mathbf{a}^x_{Ad}\}$, the family $\mathbf{E} \ra \PP^1 \times T$ can be used to make a $T_1$-family $(\mathbf{E}^{\tt I},\boldsymbol{\theta^{^{ext}}},\{\mathbf{a}^x_{Ad}\}) \ra \PP^1 \times T_1$ of parabolic principal $G$-bundles with extended weights. Thus we  have a degeneration to $\cE^{{\tt I}}$ where the underlying bundle of the generic object is trivial. Under  the morphism $\cM_X(\cG^{{\tt I}}) \ra \cM_X(\cG)$ let $\mathbf{E}^{\tt I} \ra \PP^1 \times T_1$ give the family $\mathbf{E_1} \ra \PP^1 \times T_1$ of $\cG$-torsors. It degenerates to $\cE$ and we view it as a family with weight $\boldsymbol{\theta^{^{ext}}}$. Further, \cite[Prop 6.1.2]{me} stability  is an open property of parahoric torsors. Thus for $t \in T_1$ generic, $\mathbf{E}_t \ra \PP^1$ is stable with weight $\boldsymbol{\theta^{^{ext}}}$. Therefore by Proposition \ref{stabilityExtendingweights} for generic $t$, $(\mathbf{E}^{\tt I},\boldsymbol{\theta^{^{ext}}},\{\mathbf{a}^x_{Ad}\})_t \ra \PP^1$ is stable.  So the trivial $G$-bundle with parabolic weights given by $\boldsymbol{\theta^{^{ext}}}$ and generic quasi-parabolic structure is stable. 
\end{proof}  
 
The above result can be interpreted in the setting of \cite{tw} with notations as in \cite[Page 716]{tw}. Let us denote the Gromov-Witten numbern as
\begin{equation} \label{GWinv} n_d(w_x | x \in \cR).
 \end{equation}   
It counts the  number of regular maps $\phi: X \ra G/P$ of degree $d$ such  that  for $x \in \cR$, $\phi(x)$ lies in a generic translate of the Schubert variety $ Y_{w_x} \subset G/P$ corresponding to $w_x \in W$. By Remark \ref{ustotw}, for computing degree of extended weights parahoric torsors we may switch from our definition to \cite{tw} including the far wall. Then by repeating the arguments exactly as in \cite[Page 741, (13)]{tw}, we get the following checkable corollary.




\begin{Cor} \label{stablepolytope} 
The polytope $\Delta^s$ is the set of points $\boldsymbol{\theta}$ satisfying the inequality
\begin{equation} \label{stability}
\sum_{x \in \cR} (w_x \omega_P, \theta_x) < d
\end{equation}
for all maximal parabolic subgroups $P \subset G$ and non-negative integers $d$ such that the Gromov-Witten invariant (cf (\ref{GWinv})) $n_d(w_x| x \in \cR) \neq 0$.
\end{Cor}

\section{Filtrations associated to Parabolic vector bundles} \label{rnfiltrations} For simplicity, we will first assume that we are working with one parabolic point $x$.
Recall a quasi-parabolic structure is giving a possibly partial filtration
\begin{equation} \label{filtration}
\cdots \supset E=F_0(E) \supset F_1(E) \supset \cdots \supset F_{l-1}(E) \supset F_{l}(E)=E(-D)  \supset \cdots
\end{equation} 
by subsheaves, which can be continued infinitely in both directions. Here $l$ is called the length of the filtration. It can at most be the rank of $E$.
It is called a parabolic sheaf if it has a system of weights $\alpha_0,\cdots,\alpha_{l-1}$ such that 
\begin{equation} \label{strictversion} 0 \leq \alpha_0 < \alpha_1 \cdots < \alpha_{l-1} <1.
\end{equation} The weight $\alpha_i$ is called the weight of the subsheaf $F_i(E)$. A given filtration (\ref{filtration}) need not be full. By choosing any complete flags for a given parabolic bundle,
in \cite{belkale} the notion of $\RR$-filtration of \cite{maruyamayokogawa} is extended to complete flag parabolic vector bundles with extended weights in \cite[Appendix]{belkale} as follows. He considers a filtration \begin{equation} \label{nfiltration} \cdots \supset E_n \supset E_{n+1} \supset \cdots
\end{equation} of sheaves with strict inclusions as in (\ref{filtration}) parametrized by $\mathbb{Z}$ together with weights $\{\alpha_n \}$  in $\RR$, which are allowed to coincide now. Thus if we forget sheaves $E_m$ for which there exists a $k>0$ such that  $\alpha_{m-k}=\alpha_m$, then the reduced subset of  $\{E_n\}$, together with the  corresponding weights which are now distinct correspond to (\ref{filtration}). In other words, (\ref{strictversion}) has become non-strict and is extended by
\begin{equation} \label{indexing2} \alpha_{k+ml}= \alpha_k+m.
\end{equation}
{\it The constructions in \cite{belkale} extend in an obvious way to mutliple parabolic points as well as to non-complete flag parabolic vector bundles, just that the indices are harder to write because there may be jumps because of partial flags.}  Notice that as subsheaves become smaller, their weights become larger. On any term $E_m$ of the filtration (\ref{filtration}) we can induce the structure of a parabolic vector bundle by using the  $l$ successive subsheaves in (\ref{filtration}) to get the flags; their corresponding weights may  lie outside of $[0,1)$, but, after sliding to make the weight of $E_m$ as zero they will lie in $[0,1)$. This will be denoted $E_{m*}$. Conversely, $E_{m*}$ gives (\ref{filtration}) upto shifting indices, and the same weights upto sliding.

\subsubsection{Key takeaway on degree computation of sub-bundles from \cite{belkale}} \label{pardegsubsheaf} For simplicity, we first work in the setting of \cite{belkale} which involves one parabolic point. Now weights may now lie outside $[0,1)$. We will denote this as $E_{m*}$. The weights are defined by (\ref{nfiltration}) and (\ref{indexing2}) in a way that the parabolic degree of $E_{m*}$ becomes independent of $m \in \ZZ$ (cf \cite[page 83 last para]{belkale}). Further, on \cite[page 84]{belkale} for any $m_1,m_2 \in \ZZ$ a natural procedure is explained to go from sub-bundles of $E_{m_1*}$ to $E_{m_2*}$. By \cite[Lemma 8]{belkale} this procedure preserves parabolic degree of sub-bundles too. Thus $E_{m_1*}$ is (semi)stable if and only if $E_{m_2*}$ is (semi)stable. These results  generalize to multiple parabolic points also.
 
\subsubsection{Interpretation of passage from $E_{m_1*}$ to $E_{m_2*}$  in our set-up of alcoves, weights, facets and Diagram \ref{Heckediag}}  \label{belkaleus} To enable us to adapt aspects of this process in the setting parahoric torsors, we need to interpret it in the language of alcoves. 

Let us consider the case of $SL(n)$. Let us label the vertices of $\mathbf{a}_0$ by integers $\{0,\cdots,n-1\}$. Any facet $\sigma$ in $\cA_T$ of dimension $d$, determines a set of $d+1$ vertices. Let us call {\bf the far wall of $\sigma$} as the codimension one facet determined by forgetting the smallest vertex. Define alcove $\mathbf{a_{k+1}}$ inductively by reflecting the alcove $\mathbf{a_{k}}$ along the far wall and label the new vertex by $n+k$. Let us view the weights of $E_{k*}$ (\ref{nfiltration}) as a point in $\ol{\mathbf{a_k}}$ by taking bary-centric coordinates $\{\alpha_{E_{k+1}}- \alpha_{E_{k}}, \cdots, \alpha_{E_{k+n}} -\alpha_{E_{k+n-1}} \}$ . When we pass from $E_{0*}$ to $E_{1*}$, it follows from (\ref{indexing2}) that the weights of $E_{m+1*}$ are obtained by reflecting the weights of $E_{m*}$ along the far wall of $\mathbf{a}_{m}$. Thus in terms of barycentric coordinates as a set they remain the same, just that their indexing is shifted by $-1 \, (mod \, n)$ respectively.

Let $\sigma$ be a facet in the closure of $\mathbf{a}_0$ of codimension one where only affine root  $\alpha_d$ vanishes. Then the morphism $\cM_X(\cG^{I}) \ra \cM_X(\cG^{\sigma})$ corresponds to forgeting subsheaves in the $\ZZ$-filtration whose index is $d~(\text{mod}~n)$. In terms of complete flag parabolic vector bundles, this corresponds to forgetting exactly one flag for $d \neq 0$ and a Hecke-modification by $E_0/E_1$ for $d=0$. These facts are of course much more general than $\mathbf{a}_0$ and its facets. They hold for any pair of facet $\sigma_1$ and its codimension one subfacet $\sigma$.  Going to far wall of $\sigma_1$ corresponds to a Hecke-modification by a sky-scraper sheaf while forgetting other vertices corresponds merely to forgetting flags in $
 \mathcal{M}_X(\mathcal{G}^{\sigma_1})$.  The above process also generalizes to the graph of the hyperplane structure suitably. More precisely, for any two facets $(\sigma_1,\sigma)$ where $\sigma$ lies in the closure of $\sigma_1$, the path we take to come from $\sigma_1$ to $\sigma$ is not important i.e forgetting flags and Hecke-modifications by sky-scraper sheaves commute.


\section{Extending Weights on $\cG^{{\tt I}}$ torsors} \label{sshm}
From this section onwards the results are of a general nature and so $X$ will be an {\em arbitrary Riemann surface of genus $g \geq 0$}. For simplicity of notation, we further assume that only one parabolic point $x \in X$ is fixed. For all that is done, it will be clear to the reader that the entire process can be carried individually at several points. However, for the main application the final conclusions will be made in terms multiple points on $X$.
 
 
In this section completing flags with induced weights construction of \cite{belkale} is generalized to {\it extending weights for $\mathcal{G}^{I}$}.  Recall that
similar to the definition of Mehta-Seshadri, in \cite{bs}  weights have been defined  for parahoric torsors to be points lying in the facets (cf \S \ref{pt}).
 Consider the morphism of stacks $  p: \cM_X(\cG^{{\tt I}}) \ra \cM_X(\cG)$. As in \S \ref{pt}, suppose that we are given weights ${\boldsymbol\theta}\,=\, \{\theta_x| x \in \cR \} \in (Y(T) \otimes \RR)^m$ such that $\theta_x$ lies in the facet $\sigma_x$ and $\sigma_x \subset \overline{\mathbf{a}_0}$ for all $x \in \cR$.
 
In \cite[BBP]{bbp}, given a representation $\rho: G \ra SL(V)$ and a parahoric Bruhat-Tits torsors $\cE$, the parabolic vector bundle $(\cE(V)_{_*},{\boldsymbol{\theta}}_{_V})$ has been constructed.  A priori, a naive approach would be to take the parabolic bundle $(\cE(V)_{_*},{\boldsymbol{\theta}}_{_V})$ and carry out a process as in \cite{belkale}, of deforming the underlying bundle after possibly some Hecke-modifications, to get one with a full-flag and suitable schema of weights. But the difficulty is that the new generic parabolic vector bundles  need  not come as an extension of structure group from any parahoric torsor via $\rho$. 
 
{\it In this section, given $(\cE,\bsm{\theta})$, a representation $\rho$ and the choice of a $\cG^{{\tt I}}$-torsor $\cE^{{\tt I}}$, we want to construct an associated object $(\cE^{{\tt I}}(V),{\boldsymbol{\theta}}^{{\tt I}}_{_V})$, where $\cE^{{\tt I}}(V)$ comes with a {\em quasi-parabolic structure} with a well-defined schema of weights ${\boldsymbol{\theta}}^{{\tt I}}_{_V}$.} Let us mention two problems that  arise when we try to extend the weight ${\boldsymbol\theta}$ for $\cM_X(\cG^{{\tt I}})$. Firstly, the weights $\theta_x$ only belong to the closure $\ol{\mathbf{a}_0}$ and not to the alcove $\mathbf{a}_0$ itself. So the set-up of \cite{bbp} does not apply. Secondly if we take a sequence of weights lying in $\mathbf{a}_0$ and converging to $\bsm{\theta}$, then by the construction in \cite[BBP]{bbp} even the quasi-parabolic structure of $(\cE^{{\tt I}}(V)_*,{\boldsymbol{\theta}}_V)$ is sensitive to the choice of $\boldsymbol{\theta}$. 
 
We address  these problems {\it by choosing for each $x \in \cR$ a finer $\rho$-facet $\mathbf{a}^x_\rho$ (defined below) in whose closure $\theta_x$ lies. These are defined by the requirement that all points in $\mathbf{a}^x_\rho$ under $\rho: \cA_T \ra \cA_{T_{SL(V)}}$ go to a fixed open facet of $SL(V)$, of dimension at most the dimension of $T$, in whose closure $\rho(\theta_x)$ lies.} Now the flag structure on $(\cE^{\tt I},\boldsymbol{\theta}, \{ \mathbf{a}^x_\rho \})$ may not be full since it will have at most $\dim(T)$ many distinct flags (or weights). More importantly, unlike \cite{belkale} the vector bundle underlying it may only be related to the one underlying $(\cE(V)_*,{\boldsymbol{\theta}}_V)$ by a Hecke-modification. So instead of completing flags we call this construction {\sl extending weights for parahoric torsors}. In the set-up of \cite{belkale} we have $\rho=Id$. This reflects the facts that a parabolic vector bundle determines a choice of an alcove whict itself is a finer $\rho$-facet. {\sl In the  applications of the constructions carried out here, we will mostly have to take $\rho=Ad$ and so $V=\mathfrak{g}$}.

\subsubsection{Extending weights construction for $\mathcal{G}^{\tt{I}}$-torsors with respect to a representation $\rho$} \label{Extendingweights} Let $\cE$ be a $\cG$-torsor and let $\cE^{{\tt I}}$ be a $\cG^{{\tt I}}$-torsor lying in the fiber of $\cM_X(\cG^{{\tt I}}) \ra \cM_X(\cG)$. To lighten the notation, it suffices to treat the case of one parabolic point i.e. $\cR=\{x\}$. For a representation $\rho: G \ra SL(V)$ we choose tori $T_G \subset G$ and $T_{SL} \subset SL(V)$ such that $\rho$ maps $T_G$ to $T_{SL}$. Thus we get a linear map 
\begin{equation} \label{linearapartmentmap}
\rho: \cA_T \ra \cA_{T_{SL}}
\end{equation}
between the apartments. In \cite[BBP]{bbp},  the usual definition of facet is generalized to {\it facets associated to a homomorphism $\rho$}  as follows.  By a generalized affine functional on $\cA_T$ we mean affine functionals for $G$ together with those of $\cA_{T_{SL}}$ viewed as functionals on $\cA_T$. For any point $x \in \cA_T$, let $Y_x^{g}$ denote the set of generalized affine functionals vanishing on $x$. For an integer $n\geq 0$, define
\begin{equation} \label{facetdefng}
\cH_n^g =\{x \in \cA_T | |Y_x^g|=n \}.
\end{equation}
A $\rho$-facet $\sigma$ of $\cA_T$ corresponding to a representation $\rho$ is defined to be a connected component of $\cH_n^g$ for some $n$. The dimension of a $\rho$-facet is its dimension as a real manifold. The finer facets satisfy the property that for any two weights belonging to it, the parabolic vector bundles associated to $\rho$ have the same quasi-parabolic structure.

\begin{defi} \label{partorextendedweights} Let $(\cE,{\boldsymbol{\theta}})$ be a parahoric torsor. By {\it extending weights for a $\cG^{{\tt I}}$-torsor $\cE^{{\tt I}}$ with respect to a given representation $\rho: G \ra SL(V)$} we mean the following. For each $x \in \cR$ we choose a  $\rho$-facet $\mathbf{a}^x_\rho \subset \mathbf{a}_0$ in whose closure $\theta_x$ lies. Then, given a weight $\theta_x$ we choose a sequence of weights $\theta_{x,n}$ lying in our chosen alcove $\mathbf{a}^x_\rho$ and converging to $\theta_x$.
 Thus  the quasi-parabolic structure of $(\cE^{\tt I}(V)_*,\rho(\theta_{x,n}))$ is independent of $n$ and is also independent of the choice of the limiting sequence $\{ \theta_{x,n} \}$. Keeping this quasi-parabolic structure fixed, the weight $\rho(\theta_{x,n}) \in \cA_{T_{SL}}$ equips the vector space $G^j_x$ of the flag at $x$ with a real number $\alpha^j_{x,n}$. We set 
\begin{equation}
\alpha^j_x= lim~\alpha^j_{x,n}.
\end{equation}
By linearity of (\ref{linearapartmentmap}), this definition is independent of the choice of weights $\{\theta_{x,n}\}$ and depends only on $\mathbf{a}^x_\rho$ and $\theta_x$. With notations as above, consider the parabolic vector bundle $(\cE^{\tt I}(V)_*,\rho({\boldsymbol{\theta_n}}))$ associated to ${\boldsymbol{\theta_n}}=\{ \theta_{x,n} \}_{x \in \cR}$. We endow the underlying quasi-parabolic vector bundle with weights $\{\alpha^j_x \}$ and denote this parabolic vector bundle or extended weight torsor as $(\cE^{{\tt I}},{\boldsymbol{\theta}}, \{\mathbf{a}^x_\rho\})$.
\end{defi}

\begin{rem} Say $\cR$ is a single point. A generic point $\boldsymbol{\theta}$ in the interior of the Weyl alcove $\mathbf{a}^x$, lies in a single $\rho$-facet $\mathbf{a}^x_{\rho}$. In this case, $(\cE^{{\tt I}},{\boldsymbol{\theta}}, \{\mathbf{a}^x_\rho\})$ is simply the associated construction $(\cE^{\tt I}(V)_*, \boldsymbol{\theta})$ of \cite{bbp}. Completing flag of \cite{belkale} is extending weights for the case $\theta$ lies in $\overline{\mathbf{a}_0} \setminus \{ \text{far wall} \}$. This is enough to determine the semistability polytope $\Delta^{ss}$.
\end{rem}
\begin{ex} In the context of \cite{belkale}, we have $G=\SL_n$, $\rho=\Id$ and so $\mathbf{a}^x_\rho$ is an alcove of $\SL_n$.  For simplicity let $\cR=\{x\}$. Let $\theta_d$ be the vertex of the alcove where only the affine root $\alpha_d$ does not vanish for $0 \leq d<n$. For $G=\SL_2, \rho=\Id$ and only one parabolic point $x$ consider a weight $\theta \in [0,1)$ and the parabolic vector bundle $V_*$ given by $\cO_{X}^{\oplus 2}$ with one flag at $x$ of weight $\theta$. Doing the extending weight construction for  the pair $(\theta_1, \mathbf{a}^x_\rho:=\mathbf{a}_0)$, taking limit as $\theta$ tends to $\theta_1$, the flag acquires weight one.  For $G=\SL_n, \rho=\Id$  weights $(b_0,b_1 \cdots ,b_n)$ in barycentric coordinates correspond to weights $( 0, b_1, b_1+b_2, \cdots, b_1+\cdots+b_n)$. In particular, for $\mathbf{a}^x_\rho=\mathbf{a}_0$, $\theta_d$ corresponds to vector bundle of degree $-d$. Further, for $d \geq 1$, $\rho(\theta_d)$ does lie on the far wall. Let us constrast with the case when $\rho(\theta_d)$ does {\bf not} lie on the far wall. The extended weight torsor $(\cE^{{\tt I}},{\boldsymbol{\theta}}, \{\mathbf{a}^x_\rho\})$ corresponds to a vector bundle $V \ra X$ of determinant $-d$, and full-flag $0 \subset F^1_x \subset F^2_x \subset \cdots F^n_x=V_x$ at $V_x$ with extended weights $d/n$ and $\rho(\theta_d)$ is {\bf not} on the far wall of $\mathbf{a}^x_\rho$. This happens because the flags are of $V$ which is not a principal $\SL_n$-bundle.
\end{ex}

\begin{rem} We illustrate sliding of weights over the example of $\cR=\{x\}$, $G=SL_3$, $\rho=Id$. Let  $\theta \in \mathbf{a}_0$ tend to a point $\theta_1$ on the far wall with barycentric coordinates $(0,b_1,b_2)$. Hence $b_2=1-b_1$. Doing the extending weight construction for the pair $(\theta_1,\mathbf{a}^x_\rho:=\mathbf{a}_0)$, we see that rank three vector bundle with full flags acquire weights $(0,b_1,b_1+b_2)$. On the other hand, $\theta_1$ corresponds to rank three vector bundles with a single flag of dimension two with weight $1-b_1$. This corresponds to the fact that $\theta_1=(b_1,b_2)$ in the barycentric coordinates of the far wall. The general case of sliding of weights is only notationally harder to write.
\end{rem}

\begin{rem}    We acquire weight one exactly when $\rho(\theta_x)$ lies on the far wall of the facet corresponding to $\mathbf{a}^x_\rho$. 
\end{rem}

\begin{ex} Let $V_*$ be as above. Now $Ad: \SL_2 \rightarrow \SL_3$ corresponds to $V_* \mapsto \Sym^2(V_*)$. Now $\Sym^2(V_*)$ is the PVB with underlying bundle $\cO_X^{\oplus 3}$  with weights $\{\theta, 2 \theta\}$ if $\theta \in [0,1/2)$ or $\cO_{X}^{\oplus 2} \oplus \cO_X(x)$ with weights $\{2\theta-1, \theta \}$ if $\theta \in [1/2,1)$. When $\theta \in [0,1/2)$, then $\Sym^2(V_*)$ corresponds to the alcove $\mathbf{a}_0 \setminus \{\text{far wall} \}$ and thus does not have a map forgetting the flags to torsors on the far wall. When $\theta \in [1/2,1)$, then the underlying degree of $\Sym^2(V_*)$ is not congruent to zero modulo three, and it corresponds to a parahoric $\SL_3$ torsor which maps to torsors on the far wall.  Let $W=\cO_X \oplus \cO_X(x)$. It corresponds to the far wall of $\SL_2$ but $\Sym^2(W)=\cO_X \oplus \cO_X(x) \oplus \cO_X(2x)$ corresponds to an affine Weyl group translate of the origin of $\SL_3$ i.e $\Sym^2(W) \otimes \cO_X(-x)$ is a principal $SL_3$-bundle. We see that $\Sym^2(W)$ and the bundle underlying $\Sym^2(V_*)$ are related by a Hecke-modification when $\theta \in [1/2,1)$.  In this sense as $\theta$ tends to $1$, $V_*$ tends to $W$. This observation is formalized in the proposition below. 
\end{ex}

\begin{prop} \label{ulvbheckemod} The vector bundles underlying $(\cE^{{\tt I}}(V),\rho(\bsm{\theta}))$ and $(\cE(V),\rho(\bsm{\theta}))$ are related by a Hecke-modification and are comparable under inclusion.
\end{prop}
\begin{proof} By construction, the quasi-parabolic structure of $(\cE^{{\tt I}}(V),\rho(\bsm{\theta}))$ is the same as that of $(\cE^{{\tt I}}(V),\rho(\bsm{\theta_n}))$. Notice that for each $x \in \cR$ the weights $\{\rho(\theta_{x,n})\}$ lie in a fixed facet $\sigma_{SL}^x$ of $SL(V)$ in whose closure $\rho(\theta_{x})$ lies. So considering the quasi-parabolic structures associated to these weights, we are in the situation of a {\it single Hecke-modification morphism} between the  stacks associated to $\{\sigma^x_{SL}\}_{x \in \cR}$ and  $\bsm{\theta}$ as in Diagram  \ref{Heckediag}.  Equivalently we have a morphism of stacks of quasi-parabolic vector bundles 
\begin{equation} \label{qpvb} QPVB(\rho(\theta_{x,n}), x \in \cR) \ra QPVB(\rho(\theta_x),x \in \cR).
\end{equation}
Under this morphism, the underlying vector bundles are related by a Hecke modification (or are the same if the morphism (\ref{qpvb}) corresponds merely to forgetting flags) and are comparable under inclusion. 
\end{proof}


\begin{rem} \label{refinement}
In the following proposition we show that the filtration (\ref{filtration}) of $(\cE(V)_*,\rho(\bsm{\theta}))$ is refined by that of $(\cE^{{\tt I}},{\boldsymbol{\theta}}, \{\mathbf{a}^x_\rho\})$ upto (possibly) shifting of indices, the weights are also preserved upto sliding by a real number and it forgets only those sheaves which do not matter for parabolic degree computations of these bundles as well as their sub-bundles as it happens in (\ref{pardegsubsheaf}).   We assume that $\cR=\{x\}$ because the following argument can be applied point by point in the general case.
\end{rem}

\begin{prop} \label{repweightRfil} Assume $\cR=\{x\}$. Let $U$ (resp $U^{\tt I}$) be the vector bundles underlying  $(\cE(V)_*,\rho(\bsm{\theta}))$ (resp. $(\cE^{{\tt I}},{\boldsymbol{\theta}}, \{\mathbf{a}^x_\rho\})$). Consider the possibly partial filtration (\ref{filtration}) of $(\cE^{{\tt I}},{\boldsymbol{\theta}}, \{\mathbf{a}^x_\rho\})$ along $x$. Then $U$ is the $m$-term for $0 \leq m \leq$ the dimension of the facet of $\SL(V)$ containing $\rho(\mathbf{a}^x_\rho)$. If we forget sheaves for which there is a bigger sheaf with the same weight, then if we shift indices and slide weights so that $U$ has index and weight zero, then we get the (\ref{filtration}) of $(\cE(V)_*,\rho(\bsm{\theta}))$ along $x$.
\end{prop}
\begin{proof}  The underlying bundle is the $0$-th term of (\ref{filtration}). Further $U$ is a term in the infinite filtration of $(\cE^{{\tt I}},{\boldsymbol{\theta}}, \{\mathbf{a}^x_\rho\})$ because it is related to $U^{\tt I}$ by a Hecke-modification while being comparable to it under inclusion by Proposition \ref{ulvbheckemod}. It is the $0$-term if and only if $U$ is obtained from $U^{\tt I}$ {\it only by forgetting flags}, but if $U$ is obtained by a non-trivial Hecke-modification of vector bundles, then $U$ will be  the $m$-th term for $m \geq 1$ as described above because $U$ may correspond to any flag at $x$ and the flag length $l_x$ is given by the dimension of the facet of $\SL(V)$ containing $\rho(\mathbf{a}^x_\rho)$.  We slide the weights to make $U$ have weight zero. Under the morphism $\cM_X(\mathcal{G}^{\mathtt I}) \rightarrow \cM_X(\cG)$ for each $x \in \cR$ depending on the facet in $\cA_{T_{SL(V)}}$ containing $\rho(\theta_{x})$ we forget the corresponding terms in the filtration. These are exactly the sheaves in the infinite filtration for which there is a bigger sheaf with the same weight. In particular, we forget all subsheaves of $U$ containing $U(-x)$ which are different from these and which have weight zero or one. Now $U_*$ has weights lying in $[0,1)$ because the formation of $U_*$ ignores $U(-x)$. Thus $U_*$ is the parabolic vector bundle corresponding to $(\cE(V)_*,\rho(\bsm{\theta}))$. Further, we get the filtration corresponding to 
$(\cE(V)_*,\rho(\bsm{\theta}))$ if we shift indices so that the index of $U$ becomes zero.  
\end{proof}

\section{(Semi)-stability of extended weight torsors} \label{ssrwt}
The aim of this section is to compare the 
 (semi)stability of $(\cE, \boldsymbol\theta)$ and that of $(\cE^{{\tt I}},{\boldsymbol{\theta}}, \{\mathbf{a}^x_\rho\})$ as defined in \eqref{partorextendedweights}. We work with $\rho = Ad$ in this section. For each $x \in \cR$, we will fix an $Ad$-facet $\mathbf{a}^x_{Ad}$ in whose closure $\theta_x$ lies.
 


\subsubsection{Definition}\label{parstab} Let us recall  the definition of (semi)stability from \cite[\S 6.1 BBP]{bbp} for a parahoric torsor $(\cE,{\boldsymbol{\theta}})$. Let $Ad: G \ra \SL(\gfr)$ be the ``adjoint" representation. Let $\cE(\gfr)_* = (\cE(\gfr), \boldsymbol\theta_{_{\gfr}})$ denote the associated parabolic vector bundle (\cite[\S 5]{bbp}). Let $\cE(\cG)$ denote the adjoint group  scheme of $\cE$ obtained by taking the quotient of $\cE \times \cG$ by the right $\cG$-action on $\cE$ and the left $\cG$ action by conjugation on itself.  The Lie algebra bundle $\text{Lie}(\cE(\cG))$ is given the structure of a parabolic Lie algebra bundle by identifying $\text{Lie}(\cE(\cG))$ with the vector bundle underlying $\cE(\gfr)_*$ (\cite[\S 6]{bbp}). 


Let $\eta$ be the generic point of the curve $X$. Let $\cE(\cG)_{\eta}$ denote the restriction of $\cE(\cG)$ to $\eta$. Let $\cP_\eta \subset \cE(\cG)_\eta$ be a parabolic subgroup scheme.  Taking the flat closure of $\cP_\eta$ in $\cE(\cG)$ we get the subgroup scheme $
\cP \subset \cE(\cG)$. The Lie algebra bundle $\text{Lie}(\cP)$ is a sub-bundle of $\text{Lie}(\cE(\cG))$, and we give it the {\it canonical induced parabolic structure} to get a parabolic Lie subbundle $\text{Lie}(\cP)_*$ of $\cE(\gfr)_*$. We recall \cite[Defn 6.2]{bbp}
\begin{defi} One calls a parahoric torsor $(\cE,{\boldsymbol{\theta}})$ (semi)stable if for every generic parabolic subgroup scheme $\cP_\eta \subset \cE(\cG)_\eta$ as above, we have  
\begin{equation*} 
\text{pardeg} (\text{Lie}(\cP)_*) \leq 0 ~~~ (\text{respectively}~~~ \text{pardeg} (\text{Lie}(\cP)_*) < 0).
\end{equation*}
\end{defi}

We now want to extend the definition of (semi)stability for extended weights. 
Let $(\cE,{\boldsymbol{\theta}})$ be a parahoric torsor. Under $p: \cM_X(\cG^{{\tt I}}) \ra \cM_X(\cG)$, let $\cE^{{\tt I}}$ map to $\cE$. Let $\cE^{{\tt I}}(\cG^{{\tt I}})$ denote the quotient of $\cE^{{\tt I}} \times \cG^{{\tt I}}$ by the right action of $\cG^{{\tt I}}$ on $\cE^{{\tt I}}$ and the left action of $\cG^{{\tt I}}$ on itself by conjugation. 

\begin{prop} The vector bundle underlying $(\cE^{{\tt I}},\boldsymbol{\theta_\gfr}, \{\mathbf{a}^x_{Ad} \})$ identifies naturally with the Lie algebra bundle $\text{Lie}(\cE^{{\tt I}}(\cG^{{\tt I}}))$ of $\cE^{{\tt I}}(\cG^{{\tt I}})$.
\end{prop}
\begin{proof} For simplicity, we may suppose that $\cR=\{x\}$. Now we take a sequence of weights $\theta_{x,n}$ lying in our chosen alcove $\mathbf{a}^x_{Ad}$ and converging to $\theta_x$.  By definition, the quasi-parabolic structure  of $(\cE^{{\tt I}},\boldsymbol{\theta_\gfr}, \{\mathbf{a}^x_{Ad} \})$ and $(\cE^{{\tt I}}(\gfr)_*,Ad(\theta_{x,n}))$ are same. By construction, the vector bundle underlying $(\cE^{{\tt I}}(\gfr)_*, Ad(\theta_{x,n}))$ is independent of $n$. Further, since $\theta_{x,n}$ lies inside $\mathbf{a}_0$, it identifies with $\text{Lie}(\cE^{{\tt I}}(\cG^{{\tt I}}))$ by the well-known identification (cf (\cite[\S 6]{bbp})).
\end{proof}
We induce the parabolic structure on $(\cE^{{\tt I}},\boldsymbol{\theta_\gfr}, \{\mathbf{a}^x_{Ad} \})$ to $\text{Lie}(\cE^{{\tt I}}(\cG^{{\tt I}}))$ and denote it by $\text{Lie}(\cE^{{\tt I}}(\cG^{{\tt I}}))_*$. Let us consider a generic parabolic subgroup scheme $\cP_{\eta} \subset \cE^{{\tt I}}(\cG^{{\tt I}})_{\eta}$. We take the flat closure of $\cP_\eta$ in $\cE^{{\tt I}}(\cG^{{\tt I}})$ to get a subgroup scheme $\cP^{\tt I} \subset \cE^{{\tt I}}(\cG^{{\tt I}})$. Consider the Lie algebra subbundle $\text{Lie}(\cP^{\tt I}) \subset \text{Lie}(\cE^{{\tt I}}(\cG^{{\tt I}}))$.  We endow $\text{Lie}(\cP^{\tt I})$ with the {\it canonical induced parabolic structure} and denote the associated parabolic vector bundle by $\text{Lie}(\cP^{\tt I})_*$.
 
\begin{defi} \label{defnExtendingsemistability} Let $(\cE,{\boldsymbol{\theta}})$ be a parahoric torsor. Under $p: \cM_X(\cG^{{\tt I}}) \ra \cM_X(\cG)$, let $\cE^{{\tt I}}$ map to $\cE$. We say that the parahoric torsor with extended weights $(\cE^{{\tt I}}, {\boldsymbol{\theta}},\{\mathbf{a}^x_{Ad} \})$ is (semi)stable if the following holds for every generic parabolic subgroup scheme $\cP_\eta \subset \cE^{{\tt I}}(\cG^{{\tt I}})_{\eta}$: with notations as above we have
\begin{equation}
\text{pardeg}(\text{Lie}(\cP^{\tt I})_*) \leq 0 \quad (\text{respectively} \quad \text{pardeg}(\text{Lie}(\cP^{\tt I})_*) <0 ).
\end{equation}
\end{defi}
\begin{rem} \label{ustotw} The (semi)stability Definition  \cite[Defn 2.2]{tw} is for the product of alcoves without their far walls. It may be naturally extended to the product of closed Weyl alcoves. It agrees with the Definition \ref{defnExtendingsemistability}.
\end{rem}
\begin{rem} Sliding of weights does not alter the difference of the parabolic slopes of a parabolic vector bundle and its sub-bundle. This applies in particular to parabolic torsors and extended weight torsors.
\end{rem} 
\begin{prop} \label{stabilityExtendingweights} We have $(\cE,{\boldsymbol{\theta}})$ is stable if and only if $(\cE^{{\tt I}}, {\boldsymbol{\theta}}, \{\mathbf{a}^x_{Ad} \})$ is stable as a extended weight parahoric torsor. The stability of  $(\cE^{{\tt I}}, {\boldsymbol{\theta}}, \{\mathbf{a}^x_{Ad} \})$ is independent of the choices of $\{\mathbf{a}^x_{Ad} \}$. 
\end{prop}
\begin{proof} By Proposition \ref{ulvbheckemod}, the Lie algebra bundles $\text{Lie}(\cE^{{\tt I}}(\cG^{{\tt I}}))$ and $\text{Lie}(\cE(\cG))$, being the vector  bundles underlying $(\cE^{{\tt I}}(\gfr), \boldsymbol\theta_{_{\gfr}})$ and $(\cE(\gfr), \boldsymbol\theta_{_{\gfr}})$ respectively, are related by a Hecke-modification and they are comparable under inclusion. Sliding of weights by a real number $a$ only changes the parabolic slope by $a$. So by Remark \ref{refinement} we have $par \mu (\text{Lie}(\cE^{{\tt I}}(\cG^{{\tt I}}))_*)-par \mu (\text{Lie}(\cE(\cG))_*)=a$ where we slide weights by $a$.

We also have a natural identification of $\cE^{{\tt I}}(\cG^{{\tt I}})_\eta$ with $\cE(\cG)_\eta$. Let us consider a generic parabolic subgroup scheme $\cP_{\eta} \subset \cE^{{\tt I}}(\cG^{{\tt I}})_{\eta}= \cE(\cG)_{\eta}$. 

We have the natural identification $\text{Lie}(\cP^{\tt I})_{\eta} =\text{Lie}(\cP^{\tt I}_{\eta})= \text{Lie}(\cP_{\eta})=\text{Lie}(\cP)_{\eta}$.  The parabolic structure on  $\text{Lie}(\cP^{\tt I})_*$ is induced from $\text{Lie}(\cE^{{\tt I}}(\cG^{{\tt I}}))_*$. So from the infinite filtration of  $\text{Lie}(\cP^{\mathbf{a}})_*$ we can extract like Remark \ref{refinement} the infinite filtration of  $\text{Lie}(\cP)_*$ by forgetting only those sheaves which do not matter for parabolic degree computations.  Thus $par \mu (\text{Lie}(\cP^{\tt I})_*)-par \mu (\text{Lie}(\cP)_*)=a$.

This shows that $(\cE,{\boldsymbol{\theta}})$ is stable if and only if $(\cE^{{\tt I}}, {\boldsymbol{\theta}}, \{\mathbf{a}^x_{Ad} \})$ is stable. Now the second statement follows.
\end{proof}

\section{Some deformation theory and the Proof of Theorem \ref{mt2}} \label{deformation}

Let $\cE^{I}$ be a trivial bundle with generic $B$-structures and weight $\boldsymbol{\theta^{ext}}$  corresponding to $\boldsymbol{\theta}$ as in Theorem \ref{mt}. In \S \ref{Extendingweights}, we have explained the construction of the parabolic vector bundle $V_*$ associated to it by the $Ad$ representation.
Let $(\cE,{\boldsymbol{\theta}})$ be a parahoric $\cG$-torsor on $\mathbb P^1$ with parahoric structure on the marked points $\cR$. Recall that the first order deformations of $(\cE,{\boldsymbol{\theta}})$ are controlled by the cohomology space  $H^1(\mathbb P^1, \cE(\mathfrak g))$, where $\cE(\mathfrak g)$ is the Lie-algebra bundle underlying the parabolic bundle $\cE(\mathfrak g)_*$. 

We further observe that if $\cE$ is a torsor for the Iwahori group scheme, then there is an underlying principal $G$-bundle with standard parabolic structures with fibres $B$ at the marked points. Under these conditions, the (semi)stability conditions in \ref{parstab} can be rephrased in terms of the usual (semi)stability of principal $G$-bundles but carrying along the Iwahori structures. In other words, $\cE$ is (semi)stable if and only if for every parabolic subgroup $P \subset G$ and reduction of structure group $\cE_P$ to $P$, we have $pardeg(\cE_P(\mathfrak p)_*) \leq 0 (< 0)$, where $\cE_P(\mathfrak p)_* \subset \cE(\gfr)_*$ gets the canonical induced parabolic structure. This definition coincides with the one in \cite{tw}. 

With these notions in place, we now complete the proof of Theorem \ref{mt2}.

\begin{proof}[Proof of Theorem \ref{mt2}] $\implies$ If the point $\theta$ lies in the stable polytope, then  by Theorem \ref{mt}, the trivial bundle with generic $B$-structures $\cE^{I}$ with weights $\mu$ is stable. So it has no sub-bundles whose associated parabolic vector bundle has degree zero owing to stability. 

$\impliedby$ Since $\boldsymbol{\theta} \in \Delta^{ss}$, so by \cite{tw} the trivial bundle  with generic $B$-structures $\cE^{I}$ is $\mu$-semistable. Now the stack of $B$-structures on the trivial bundle is smooth and irreducible. Therefore, we have a smooth and irreducible versal space $T$ and let $\cE$ be a versal torsor. If at the generic point of $T$   the torsor $\cE$ is only semistable and not stable, then    
in the setting of \ref{parstab} the family $\cE_{t}(\gfr)_*$ of torsors would admit  parabolic reductions $\cE_P(\mathfrak p)_{t,*}$  with induced parabolic structures such that $\text{pardeg}(\cE_P(\mathfrak p)_{t,*}) = 0$.

At the point $t \in T$ corresponding to $\cE^{I}$ with generic $B$-structures, we consider the map from \{{\sl deformations of $\cE^{\tt I}_{_P}$}\}  to \{{\sl deformations of $\cE^{I}$}\}.  Observe that we have the following exact sequence of parabolic bundles:
\begin{equation} \label{deformextseq}
0 \to \cE^{\tt I}_{_P}(\mathfrak p)_{*} \to \cE^{I}(\gfr)_* \to \cE^{I}(\gfr)_*/\cE_P(\mathfrak p)_{*}\to 0
\end{equation}
By assumption on $\cE^{I}$, no $P$-reduction $\cE^{\tt I}_{_P}$ is of the {\sl minus 1 type}. So the quotient $\cE^{I}(\gfr)_*/\cE_P(\mathfrak p)_{*}$ has $\text{pardeg} = 0$  but the degree of the underlying vector bundle is {\bf not $-1$}. Hence it has a non-trivial $H^1$. This implies that  the the map \{{\sl deformations of $\cE^{\tt I}_{_P}$}\}  to \{{\sl deformations of $\cE^{I}$}\} is {\em not surjective}.  In other words, $\cE^{\tt I}_{_P}$ does not deform to the generic point of the versal space. And this holds for all parabolic reductions $\cE^{\tt I}_{_P}$. So, at the generic point $\eta$ of $T$ the torsor $\cE_{\eta}$ has no destabilizing $\cE^{\tt I}_{P,\eta}$ arising as a deformation from a closed point of $T$. On the other hand, any $\cE^{\tt I}_{P,\eta}$ of $\cE_{\eta}$ spreads to a Zariski neighbourhood of  
$\eta$ in $T$. So $\cE$ must be {\sl stable}. Hence $\boldsymbol{\theta} \in \Delta^{s}$.


\end{proof}
\begin{Cor}\label{newvercrit} The above theorem gives a verifiable criteria for $\Delta^{s}$.
\end{Cor}
\begin{proof} For parabolic vector bundles, one considers the trivial bundle with generic quasi-parabolic structures of a fixed type. Then Gromov-Witten input data exactly corresponds to the triples of rank, degree and type of induced parabolic structure on sub-bundles. Recall that this generalizes suitably also for principal $G$-bundles as in \cite{tw}.  Consider any $P$-reduction of structure group $\cE^{\tt I}_{_P}$ of $\cE^{\tt I}$. By associated constructions, it leads to a PVB sub-bundle like $\cE_P(\mathfrak p)_{*}$ in (\ref{deformextseq}). Notice that
the local contributions to the parabolic degree of $\cE_P(\mathfrak p)_{*}$ by weights gets determined. Further, if $\cE_P(\mathfrak p)_{*}$ has the same slope as $\cE^{I}(\gfr)_*$, then its underlying degree gets determined as well. Thus, the set of input data $(d,w_x|x \in \cR)$ for Gromov-Witten numbers of sub-bundles that could potentially violate stabilty gets determined. These input data also determine the underlying degree of the quotient bundle $\cE^{I}(\gfr)_*/\cE_P(\mathfrak p)_{*}$. Now,   the condition ``no $P$-reduction $\cE^{\tt I}_{_P}$ is of the {\sl minus 1 type}"  can be formulated in terms of vanishing of Gromov-Witten numbers.
\end{proof}
\begin{rem} The above elementary deformation arguments can be used with the Harder-Narasimhan parabolic reduction to cut down on the input data of Gromov-Witten numbers for $\Delta^{ss}$.  \end{rem}


\bibliographystyle{plain}

\bibliography{stability}
 \end{document}